\documentclass[11pt]{article}
\usepackage{latexsym,amsmath,amssymb}

\newif\ifdviwin

\usepackage[center]{titlesec}
\usepackage[english]{babel}
\usepackage{indentfirst}
\usepackage[mathscr]{eucal}
\usepackage{amssymb,amsmath,amsfonts}
\usepackage{fancybox,fancyhdr}
\usepackage{graphicx}
\usepackage[utf8]{inputenc}
\usepackage{float}
\usepackage{color}
\usepackage[pdftex]{hyperref}
\hypersetup{colorlinks=true,linkcolor=blue,citecolor=blue} 
\usepackage[export]{adjustbox}

\newif\ifdviwin

\dviwintrue

\def\fC{\mathfrak{C}}

\def\ch{\mathfrak{h}}

\def\m2r{\mathbb{M}^2\times\mathbb{R}}
\def\h2r{\mathbb{H}^2\times\mathbb{R}}

\let\8=\infty \let\0=\emptyset

\def\cte.{\mathop{\rm cte.}\nolimits}

\def\N{\mathbb{N}}

\def\R{\mathbb{R}}
\def\F{\mathcal{F}}
\def\M{\mathbb{M}}
\def\C{\mathcal{C}}
\def\H{\mathcal{H}}

\def\m2r{\M^2\times\R}

\def\h2r{\mathbb{H}^2\times\R}
\def\s2r{\mathbb{S}^2\times\R}

\def\hi{\H\in\c1}

\def\Hs{\mathcal{H}\text{-}\mathrm{surface}}
\def\Hss{\mathcal{H}\text{-}\mathrm{surfaces}}

\def\sig{\Sigma}
\def\r3{\mathbb{R}^3}

\def\c1{\fC^1([-1,1])}

 \newtheorem{defi}{Definition}[section]
 \newtheorem{teo}[defi]{Theorem}
 \newtheorem{pro}[defi]{Proposition}
 
 \newtheorem{lem}[defi]{Lemma}
 
 \newtheorem{remark}[defi]{Remark}
 
 \newtheorem{obs}[defi]{Observation}

 \newenvironment{proof}{\rm \trivlist \item[\hskip \labelsep{\it
      Proof}:]}{\nopagebreak \hfill $\Box$ \endtrivlist}

\numberwithin{equation}{section}

\begin{document}

\mbox{}\vspace{0.4cm}

\begin{center}
\renewcommand{\thefootnote}{\,}
{\Large \bf Half-space theorems for properly immersed surfaces in $\r3$ with prescribed mean curvature}
\footnote{\hspace{-.75cm}
The author was partially supported by MICINN-FEDER Grant No. MTM2016-80313-P, Junta de Andalucía Grant No. FQM325 and FPI-MINECO Grant No. BES-2014-067663.}\\
\vspace{0.5cm} { Antonio Bueno}\\
%\rule{16cm}{1.5pt}
\end{center}
\vspace{.5cm}
Departamento de Geometría y Topología, Universidad de Granada, E-18071 Granada, Spain. \\ 
\emph{E-mail address:} jabueno@ugr.es \vspace{0.3cm}

\begin{abstract}
Motivated by the large ammount of results obtained for minimal and positive constant mean curvature surfaces in several ambient spaces, the aim of this paper is to obtain half-space theorems for properly immersed surfaces in $\r3$ whose mean curvature is given as a prescribed function of its Gauss map. In order to achieve this purpose, we will study the behavior at infinity of a 1-parameter family of properly embedded annuli that are analogous to the usual minimal catenoids. 
\end{abstract}
 
\section{\large Introduction}
\vspace{-.5cm}

One of the most beautiful theorems in the theory of immersed minimal surfaces in $\r3$ is the \emph{half-space theorem} due to Hoffman and Meeks \cite{HoMe}, which can be formulated as follows:
\begin{teo}[Half-space theorem]
 A connected, proper, possibly branched, nonplanar minimal surface in $\r3$ cannot be contained in a half-space.
\end{teo}
Their proof is based on a clever application of two properties that minimal surfaces in $\r3$ satisfy:

\begin{itemize}
\item[1.] The coordinates of a minimal surface in $\r3$ are harmonic and thus minimal surfaces satisfy the \emph{tangency principle}: two minimal surfaces cannot be tangent in an interior point.

\item[2.] For every plane $\Pi\subset\r3$ and each line $L$ orthogonal to $\Pi$, there exists a 1-parameter family of properly embedded minimal annuli $\{\mathcal{C}(r)\}_{r>0}$ that are rotationally symmetric around $L$, and such that: \textbf{i)} $\{\mathcal{C}(r)\}_{r>0}$ smoothly converges to a double covering of $\Pi-(\Pi\cap L)$ when $r\rightarrow 0$, and \textbf{ii)} for each fixed $r_0>0$, $\mathcal{C}(r_0)$ is a symmetric bi-graph over $\Pi$ with both components having unbounded height w.r.t. $\Pi$.
\end{itemize}
The 1-parameter family of properly embedded minimal annuli $\{\mathcal{C}(r)\}_{r>0}$ are the minimal catenoids. Note that this theorem is not true in an Euclidean space $\R^n$ of arbitrary dimension, since minimal catenoids of dimension $n-1$ for $n>3$ are contained between two parallel planes.

Inspired by the ideas developed by Hoffman and Meeks, the existence of half-space theoerms in a class of immersed surfaces in several ambient spaces have attracted the attention of a large amount of geometers, becoming an active and fruitful field of research. In \cite{RoRo}, Rosenberg and Rodriguez obtained a half-space theorem for constant mean curvature one surfaces in the hyperbolic three space $\mathbb{H}^3$. Their proof is based in the original ideas of Hoffman and Meeks and in the fact that minimal surfaces in $\r3$ are locally isometric to constant mean curvature one surfaces in $\mathbb{H}^3$. Later, Hauswirth, Rosenberg and Spruck \cite{HRS} obtained a half-space theorem for constant mean curvature $1/2$ surfaces in the product space $\mathbb{H}^2\times\R$, and they exploited it to prove that complete multigraphs in $\mathbb{H}^2\times\R$ are indeed entire graphs over the whole hyperbolic plane $\mathbb{H}^2$. Finally, Daniel and Hauswirth \cite{DaHa} obtained half-space theorems for minimal surfaces in the Lie group $\mathrm{Nil}_3$, the Heisenberg space, and Daniel, Meeks and Rosenberg \cite{DMR} proved half-space theorems in $\mathrm{Nil}_3$ and also in the Lie group $\mathrm{Sol}_3$.

Motivated by these results, our purpose in this paper is to obtain half-space theorems for the following class of immersed surfaces in $\r3$: let be $\H\in C^1(\mathbb{S}^2)$. We say that an immersed surface $\sig$ in $\r3$ has \emph{prescribed mean curvature $\H$} if the mean curvature $H_\sig$ of $\sig$ satisfies at each $p\in\sig$
\begin{equation}\label{defHsupintro}
H_\sig(p)=\H(\eta_p),\hspace{.5cm} \forall p\in\sig,
\end{equation}
where $\eta:\sig\rightarrow\mathbb{S}^2$ is the \emph{Gauss map} of $\sig$. For short, we will say that $\sig$ is an $\H$-\emph{surface}. 

The definition of this class of immersed surfaces in $\r3$ has its origins in the famous Christoffel and Minkowski problems for ovaloids, see e.g. \cite{Chr}. In particular, the existence and uniqueness of ovaloids with prescribed mean curvature \eqref{defHsupintro} was studied among others by Alexandrov and Pogorelov \cite{Ale,Pog}. Besides the milestones reached concerning the uniqueness of ovaloids with prescribed mean curvature, the global properties of immersed surfaces in $\r3$ governed by Equation \eqref{defHsupintro} remained largely unexplored until Bueno, Gálvez and Mira \cite{BGM} developed the global theory of surfaces with prescribed mean curvature. In their paper, they covered topics such as the existence and classification of rotational surfaces, existence of a priori height and curvature estimates, stability properties, non-existence of complete stable surfaces and classification of properly embedded surfaces with at most one end. See also \cite{Bue1} for the resolution of the Björling problem for $\Hss$ in $\r3$ and \cite{Bue2,Bue3} for an extension of this theory to the product spaces $\mathbb{M}^2(\kappa)\times\R$.

The rest of the introduction is devoted to highlight the organization of the paper:

In \textbf{Section \ref{sec2}} we recall some basic properties of $\Hss$ in $\r3$. Locally, $\Hss$ are governed by a quasilinear, elliptic PDE, and thus they satisfy the mean curvature comparison principle and the maximum principle, see Lemmas \ref{meancurvprinc} and \ref{ppiomax}. For the particular case that the prescribed function $\H$ depends only on the height of the sphere, then it can be realized as a 1-dimensional function and Equation \eqref{defHsupintro} reads as
\begin{equation}\label{hdependealtura}
H_\sig(p)=\H(\langle\eta_p,e_3\rangle),\hspace{.5cm} \forall p\in\sig.
\end{equation}
In this situation, if $\hi$, see Equation \eqref{defhic1} for a proper definition of the space $\c1$, the study carried on in \cite{BGM} reveals that there exists a 1-parameter family of properly embedded annuli, called $\H$-catenoids, that are bi-graphs over a horizontal plane and will play the same role as minimal catenoids for our purpose. In Propositions \ref{comparisoncats} and \ref{monotoniaderivadas} we state some comparison theorems concerning the height and the derivative of the $\H$-catenoids.

In \textbf{Section \ref{sec3}} we analyze the behavior at infinity of the $\H$-catenoids, i.e. the boundness or unboundness of the heights of their graphical components.  In Proposition \ref{mismocomportamientohcats} we prove that the behavior at infinity of a family of $\H$-catenoids is determined by the behavior of an $\H$-catenoid $\sig_\H(r_0)$, for an arbitrary $r_0>0$. In Proposition \ref{hfmismocomportamiento} we relate the limit behavior of two functions $\H,\F\in\c1$ at the points $y=\pm 1$ with the behavior at infinity of two catenoids $\sig_\H(r_0)$ and $\sig_\F(r_0)$. We conclude this analysis in Theorem \ref{claseequiv}, proving that two prescribed functions $\H$ and $\F$ with the same limit behavior at the points $y=\pm 1$ determine prescribed mean curvature catenoids with the same behavior at infinity.

Bearing in mind the results obtained in Section 3, in \textbf{Section \ref{sec4}} we study the behavior at infinity of $\H$-catenoids for concrete choices of the prescribed function $\H$. Indeed, we prove in Theorem \ref{alfamayor1} that the prescribed functions $\H_\alpha(y)=-(1-y^2)^\alpha,\ \alpha>1$ generate $\H_\alpha$-catenoids with unbounded height. 

Finally, in \textbf{Section \ref{sec5}} we take advantage of the analysis made in the previous sections in order to obtain half-space theorems for properly immersed $\Hss$, provided that $\H\in C^1(\mathbb{S}^2)$ satisfy some necessary hypothesis.

\section{\large Properties of $\H$-surfaces}\label{sec2}
\vspace{-.5cm}

\begin{defi}\label{defHsup}
Let be $\H\in C^1(\mathbb{S}^2)$. An immersed surface $\sig$ in $\r3$ is an $\H$-surface if its mean curvature $H_\sig$ is given at every $p\in\sig$ by
\begin{equation}\label{hdependenormal}
H_\sig(p)=\H(\eta_p),
\end{equation}
where $\eta:\sig\rightarrow\mathbb{S}^2$ is the Gauss map of $\sig$. 
\end{defi}
It is clear from this definition that the only ambient isometries which are also isometries for the class of immersed $\Hss$ are Euclidean translations; any other ambient isometry changes the expression of $\eta$ and thus would not preserve Equation \eqref{defHsup}.

Before formulating two key properties that $\Hss$ satisfy, we need to introduce the concept of when a surface is locally above other. Let be $\sig_1$ and $\sig_2$ two immersed surfaces in the Euclidean space $\r3$. Suppose that there exists some point $p\in\sig_1\cap\sig_2$ such that $(\eta_{\sig_1})_p=(\eta_{\sig_2})_p$, where $\eta_{\sig_i}$ stands for the unit normal of the surface $\sig_i$. In this situation it is known that both $\sig_i$ can be expressed locally around $p$ as graphs $u_1,u_2$ defined in the same open set $\Omega$ of the tangent plane $T_p\sig_1=T_p\sig_2$ containing the origin $\textbf{o}$ and such that $u_i(\textbf{o})=p$.

\begin{defi}
With the previous hypothesis, we say that $\sig_1$ \emph{lies locally above} $\sig_2$ if $u_1\geq u_2$ in $\Omega$.
\end{defi}
The condition \emph{$\sig_1$ lies locally above $\sig_2$} will be written for short as $\sig_1\geq\sig_2$.

The next lemma relates the mean curvature of two surfaces lying one locally above of the other:

\begin{lem}[Mean curvature comparison principle]\label{meancurvprinc}
Let be $\sig_1,\sig_2$ two immersed surfaces in $\r3$ and denote by $H_i$ to the mean curvature of $\sig_i,\ i=1,2$. If $\sig_1\geq\sig_2$ around some $p\in\sig_1\cap\sig_2$, then $H_1(p)\geq H_2(p)$.
\end{lem}

The study made in Section 2.1 in \cite{BGM} reveals that $\Hss$ in $\r3$ are solutions of a quasilinear, second order, elliptic PDE. In particular, the class of immersed $\Hss$ satisfy the Hopf maximum principle in both its interior and boundary versions, a result that has the following geometric implication:

\begin{lem}[Maximum principle for $\H$-surfaces]\label{ppiomax}
Let be $\sig_1,\sig_2$ two immersed $\H$-surfaces in $\r3$. Assume that one of the following two conditions holds:
\begin{enumerate}
\item There exists $p\in {\rm int}(\Sigma_1)\cap {\rm int}(\Sigma_2)$ such that $(\eta_{\sig_1})_p=(\eta_{\sig_2})_p$, where $\eta_{\sig_i}$ denotes the unit normal of $\Sigma_i$, $i=1,2$.
\item There exists $p\in\partial\Sigma_1\cap\partial\Sigma_2$ such that $(\eta_{\sig_1})_p=(\eta_{\sig_2})_p$ and $(\xi_{\sig_1})_p=(\xi_{\sig_2})_p$, where $(\xi_{\sig_i})_p$ denotes the interior unit conormal of $\partial\Sigma_i$.
\end{enumerate}
Assume moreover that $\Sigma_1$ lies around $p$ at one side of $\Sigma_2$. Then $\Sigma_1=\Sigma_2$.
\end{lem}

As we mentioned in the introduction of Section \ref{sec2}, the only isometries of $\r3$ that preserve Equation \eqref{hdependenormal} are Euclidean translations. Thus, if we expect to define rotationally symmetric $\Hss$, additional symmetries have to be imposed to the prescribed function $\H$. In this fashion, if suppose that the prescribed function $\H\in C^1(\mathbb{S}^2)$ only depends on the height of the sphere, then Equation \eqref{hdependenormal} for an immersed $\Hs$ $\sig$ reads as
\begin{equation}\label{hdependeangulo}
H_\sig(p)=\H(\eta_p)=\ch(\langle\eta_p,e_3\rangle),\hspace{.5cm} \forall p\in\sig,
\end{equation}
where $\ch\in C^1([-1,1])$ and the quantity $\langle(\eta_\sig)_p,e_3\rangle$ is the so called \emph{angle function}, which will be denoted for short by $\nu_\sig(p)$.

Now, the ambient isometries that preserve Equation \eqref{hdependeangulo} are the following: Euclidean translations, the isometric $SO(2)$-action of rotations that leave pointwise fixed any vertical line, and reflections w.r.t. any vertical plane; any of these isometries leaves invariant the angle function of an immersed surface and thus preserves Equation \eqref{hdependeangulo}.

Otherwise stated, we will restrict ourselves to prescribed functions that depend only on the height of the sphere, and thus the class of $\Hss$ are governed by Equation \eqref{hdependeangulo}. For the sake of clarity, the 1-dimensional prescribed function $\ch$ that appears in Equation \eqref{hdependeangulo} will be denoted again by $\H$.

In Section 3 in \cite{BGM} the authors studied rotationally symmetric $\Hss$ for several choices of prescribed 1-dimensional functions $\H\in C^1([-1,1])$, obtaining a large amount of rotational examples with different topological properties and behaviors at infinity. In this paper we will restrict ourselves to the following class of 1-dimensional functions
\begin{equation}\label{defhic1}
\mathfrak{C}^1([-1,1]):=\{\H\in C^1([-1,1]);\ \H(y)<0,\ \forall y\in (-1,1),\ \H(-1)=\H(1)=0\}.
\end{equation}
For the particular choice $\hi$, Proposition 3.6 in \cite{BGM} proves the existence of the following family of rotationally symmetric $\Hss$:

There exists a continuous 1-parameter family $\{\sig_\H(r)\}_{r>0}$ of properly embedded, rotationally symmetric $\Hss$ around the vertical line passing through the origin, all having the topology of an annulus. For each $r_0>0$, $\sig_\H(r_0)$ is a bi-graph over the exterior of the disk $D(0,r_0)$ contained in a horizontal plane (which can be supposed to be the plane $\{z=0\}$ after a vertical translation). The annulus $\sig_\H(r_0)$ is foliated by parallel circumferences, and the smallest one is the given by the intersection $\sig_\H(r_0)\cap\{z=0\}$, which is called the \emph{waist} of $\sig_\H(r_0)$ and whose \emph{necksize}, i.e. the radius of the waist, is precisely $r_0$. In particular, each $\sig_\H(r_0)$ is contained inside $\r3-(D(0,r_0)\times\R)$. 

Each component of the bi-graph defines an end of $\sig_\H(r_0)$. We define $\sig_\H^+(r_0):=\sig_\H(r_0)\cap\{z\geq 0\}$ (resp. $\sig_\H^-(r_0):=\sig_\H(r_0)\cap\{z\leq 0\}$) the upper (resp. lower) end of $\sig_\H(r_0)$. If we denote by $\eta_{\sig_\H(r_0)}$ to the unit normal of $\sig_\H(r_0)$, then $\eta_{\sig_\H(r_0)}$ points inwards at the waist, upwards at the upper end and downwards at the lower end.

\begin{figure}[H]
\centering
\includegraphics[width=.9\textwidth]{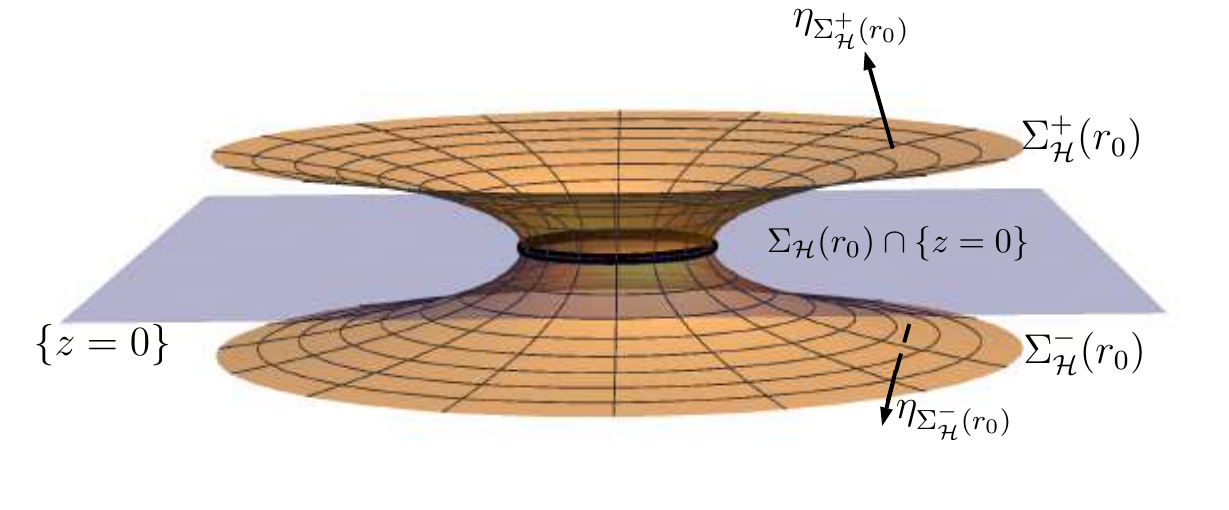}
\caption{An $\H$-catenoid for the prescribed choice $\H(y)=-(1-y^2)^2$, which is a symmetric bi-graph over the horizontal plane $\{z=0\}$. The waist is plotted in black.}
\label{homotecia}
\end{figure}

Because of the orientation that $\sig_\H(r_0)$ has, we can parametrize the upper end $\sig_\H^+(r_0)$ by rotating the graph of a function $f_{\sig_\H^+(r_0)}:(r_0,\infty)\rightarrow\R$ around the vertical axis passing through the origin. When fixing $\hi$ and $r_0>0$, in order to save notation, the function $f_{\sig_\H^+(r_0)}$ will be just denoted by $f_+$. A parametrization of $\sig_\H^+(r_0)$ minus one meridian is given by
\begin{equation}\label{pararot}
\psi_{f_+}(x,\theta)=(x\cos\theta,x\sin\theta,f_+(x)),\hspace{.5cm} x>r_0,\ \theta\in (0,2\pi).
\end{equation}

The orientation induced by $\psi_{f_+}$ is given by the unit normal 
$$
\eta_{f_+}=\frac{1}{\sqrt{1+{f_+}'(x)^2}}(-{f_+}'(x),1),
$$
which happens to be the upwards one in this upper end. For this parametrization $\psi_{f_+}$, the mean curvature $H_{\sig_\H^+(r_0)}$ of $\sig_\H^+(r_0)$ satisfies the following ODE
\begin{equation}\label{eqcurvmedia}
2H_{\sig_\H^+(r_0)}(\psi_{f_+}(x,\theta))=\frac{{f_+}''(x)}{(1+{f_+}'(x)^2)^{3/2}}+\frac{{f_+}'(x)}{x\sqrt{1+{f_+}'(x)^2}},\hspace{.5cm} \forall x>r_0.
\end{equation}
As $\sig_\H^+(r_0)$ is an $\H$-surface, Equation \eqref{defHsup} now reads as
\begin{equation}\label{condicionHsup}
2\H(\nu_{f_+}(x))=\frac{{f_+}''(x)}{(1+{f_+}'(x)^2)^{3/2}}+\frac{{f_+}'(x)}{x\sqrt{1+{f_+}'(x)^2}},\hspace{.5cm} \forall x>r_0,
\end{equation}
where $\nu_{f_+}:\sig_\H^+(r_0)\rightarrow\R$ is the \emph{angle function}
\begin{equation}\label{angulografo}
\nu_{f_+}(x)=\frac{1}{\sqrt{1+{f_+}'(x)^2}},\hspace{.5cm} \forall x>r_0.
\end{equation}

Moreover, we can solve Equation \eqref{condicionHsup} in terms of $f_+''(x)$ and conclude that is a solution of the ODE
\begin{equation}\label{ODEfsegunda}
f_+''(x)=\left(1+f_+'(x)^2\right)\left(2\H(\nu_{f_+}(x))\sqrt{1+f_+'(x)^2}-\frac{f_+'(x)}{x}\right).
\end{equation}

Let us analyze the lower end $\sig_\H^-(r_0)$. If we parametrize $\sig_\H^-(r_0)$ as in Equation \eqref{pararot} for a function $f_-$, then this parametrization does not define an $\H$-surface for the prescribed choice $\H$, since this time the unit normal induced by this parametrization, namely 
\begin{equation}\label{normalcontrario}
\eta_{f_-}=\frac{1}{\sqrt{1+{f_-}'(x)^2}}(-{f_-}'(x),1),
\end{equation}
is again the upwards one. In particular, the mean curvature with this parametrization is positive. Nonetheless, up to a change of the orientation, which in particular changes the sign of the mean curvature, this parametrization defines the lower end of $\sig_\H(r_0)$.

The functions $f_+$ and $f_-$ will be called the \emph{upper height and lower height} of $\sig_\H(r_0)$, respectively. These functions are defined in the same interval $(r_0,\infty)$, and they can be smoothly glued together at $x=r_0$, where they meet each other at the plane $\{z=0\}$ in an orthogonal way, defining the waist of the $\H$-catenoid.

%Consider now a function $\hi$, and suppose that $\sig_f$ is an $\Hs$. Thus, Equation \eqref{pararot} can be written as
%\begin{equation}\label{ecurot}
%f''(x)=\left(1+f'(x)^2\right)\left(2\H(\nu_f(x))\sqrt{1+f'(x)^2}-\frac{f'(x)}{x}\right),
%\end{equation}
%where $\nu_f:\sig_f\rightarrow\R$ is the \emph{angle function}
%$$
%\nu_f(x)=\frac{1}{\sqrt{1+f'(x)^2}}.
%$$

\subsection{Comparison of the height and the derivative of $\H$-catenoids}
In the case that we have a rotational $\Hs$, and in particular for prescribed mean curvature catenoids, the mean curvature comparison principle has the following implication.

\begin{pro}\label{comparisoncats}
Let be $\H,\F\in\c1$ and $r_0>0$, and suppose that $\H(y)>\F(y)$ for all $y\in (-1,1)$.
\begin{itemize}
\item[1.] If $h_+(x)$ and $f_+(x)$ denote the upper heights of $\sig_\H(r_0)$ and $\sig_\F(r_0)$, then $h_+(x)>f_+(x)$ for all $x>r_0$.
\item[2.] If $h_-(x)$ and $f_-(x)$ denote the lower heights of $\sig_\H(r_0)$ and $\sig_\F(r_0)$, then $h_-(x)<f_-(x)$ for all $x>r_0$.
\end{itemize}
\end{pro}

\begin{proof}
The proof will be done for the upper ends of $\sig_\H(r_0)$ and $\sig_\F(r_0)$, since the argument is similar for the lower ends.

Consider the catenoids $\sig_\F(r_0)$ and $\sig_\H(r_0)$, whose necksizes are exactly $r_0$. Both $\sig_F(r_0)$ and $\sig_\H(r_0)$ are tangent at $r_0$ with unit normals agreeing at their waists, and thus the mean curvature comparison principle ensures us that $\sig_\H(r_0)$ lies locally above $\sig_\F(r_0)$. Because $\sig_H(r_0)$ lies locally above $\sig_\F(r_0)$, we have that $h_+(x)>f_+(x)$ for $x>r_0$ close enough to $r_0$. 

Arguing by contradiction, suppose that there exists some $r_1>r_0$ such that $h_+(r_1)=f_+(r_1)$. Consider the slab $\mathcal{S}$ in $\r3$ determined by the vertical planes $\{x=r_0\}$ and $\{x=r_1\}$, and denote by $\widetilde{\sig_\H^+(r_0)}$ and $\widetilde{\sig_\F^+(r_0)}$ to the intersections of $\sig_\H^+(r_0)$ and $\sig_\F^+(r_0)$ with $\mathcal{S}$. Notice that $\widetilde{\sig_\H^+(r_0)}$ and $\widetilde{\sig_\F^+(r_0)}$ only intersect each other along their (compact) boundaries, which are contained in $\partial\mathcal{S}$.

Consider the uniparametric group of vertical translations $T_s(p)=p+(0,0,s),\ s>0$, and the translated $\Hss$ $T_s(\widetilde{\sig_\F^+(r_0)})$. Because the boundary of $\widetilde{\sig_\F^+(r_0)}$ is compact, there exists some $s_0>0$ such that $T_{s_0}(\widetilde{\sig_\F^+(r_0)})\cap\widetilde{\sig_\H^+(r_0)}=\varnothing$. Then, we decrease the parameter $s$ starting from $s_0$ until we find an instant $s_1\in (0,s_0)$ such that $T_{s_1}(\widetilde{\sig_\F^+(r_0)})$ has a first contact point of intersection $p_1$ with $\widetilde{\sig_\H^+(r_0)}$. Since $\widetilde{\sig_\H^+(r_0)}$ and $\widetilde{\sig_\F^+(r_0)}$ only intersected along their boundaries, this point of intersection must be an interior one. See Figure \ref{arribaabajo} for a diagram of this process.

\begin{figure}[H]
\centering
\includegraphics[width=.7\textwidth]{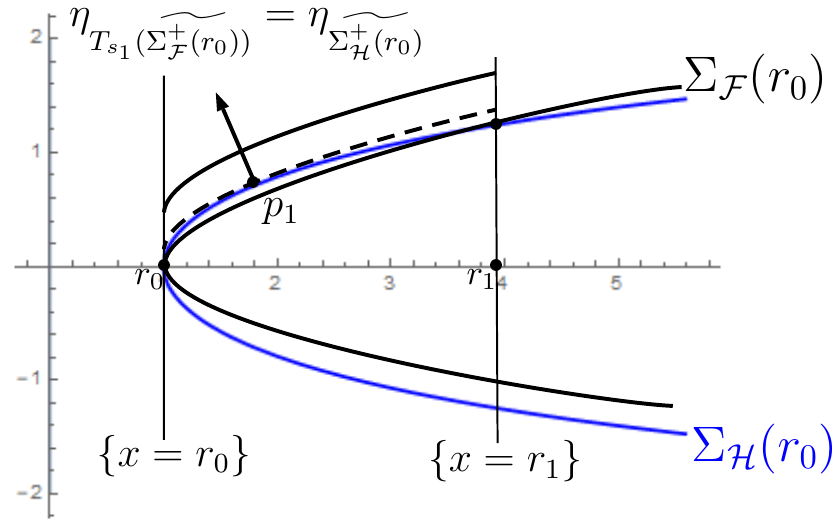}
\caption{A diagram showing the upwards and downwards movement of the compact piece $\widetilde{\sig_\F^+(r_0)}$, arriving to a contradiction.}
\label{arribaabajo}
\end{figure}

At $p_1$ the unit normals of $\widetilde{\sig_\H^+(r_0)}$ and $T_{s_1}(\widetilde{\sig_\F^+(r_0)})$ agree. Moreover, $T_{s_1}(\widetilde{\sig_\F^+(r_0)})$ lies above $\widetilde{\sig_\H^+(r_0)}$ around $p_1$, but their mean curvatures satisfy
$$
H_{T_{s_1}(\widetilde{\sig_\F^+(r_0)})}(p)=\F((\eta_{T_{s_1}(\widetilde{\sig_\F^+(r_0)})})_{p_1})<\H((\eta_{T_{s_1}(\widetilde{\sig_\F^+(r_0)})})_{p_1})=\H((\eta_{\widetilde{\sig_\H^+(r_0)}})_{p_1})=H_{\widetilde{\sig_\H^+(r_0)}}(
p_1),
$$
where we have used that $\H(y)>\F(y)$ pointwise. This is a contradiction with the mean curvature comparison principle. 

The same proof holds for the lower ends, by just considering vertical translations that decrease the height. This completes the proof of Proposition \ref{comparisoncats}.
\end{proof}

The next proposition gives us information about the derivatives of the $\H$-catenoids.

\begin{pro}\label{monotoniaderivadas}
Let be $\H,\F\in\mathfrak{C}^1([-1,1])$ such that $\H(y)>\F(y)$ for all $y\in (-1,1)$, and let be $r_0>0$ and $x_0>r_0$. Denote by $h_+,h_-,f_+,f_-$ to the upper and lower heights of $\sig_\H(r_0)$ and $\sig_\F(r_0)$, respectively. Then,
\begin{itemize}
\item[1.] If ${h_+}'(x_0)>{f_+}'(x_0)$, we have ${h_+}'(x)>{f_+}'(x)$ for all $x>x_0$.
\item[2.] If ${h_-}'(x_0)<{f_-}'(x_0)$, we have ${h_+}'(x)<{f_+}'(x)$ for all $x>x_0$.
\end{itemize}
\end{pro} 

\begin{proof}
The proof will be done for the upper heights $h_+,f_+$, since it is analogous for the lower heights. For the sake of clarity, we will drop the sub index $(\cdot)_+$, and just write $h,f$. First, recall from Equation \eqref{ODEfsegunda} that both $h$ and $f$ are solutions of the ODE's
\begin{equation}\label{hfsegunda}
\begin{array}{c}
\vspace{.5cm}h''(x)=\left(1+h'(x)^2\right)\displaystyle{\left(2\H(\nu_h)\sqrt{1+h'(x)^2}-\frac{h'(x)}{x}\right)},\\
f''(x)=\displaystyle{\left(1+f'(x)^2\right)\left(2\F(\nu_f)\sqrt{1+f'(x)^2}-\frac{f'(x)}{x}\right)},
\end{array}
\end{equation}
where $\nu_h=1/\sqrt{1+h'(x)^2}$ is the angle function of $\sig_\H^+(r_0)$, and the same holds for $\nu_f$.

Let us write
$$
\phi_\H(x,y)=\displaystyle{\left(1+y^2\right)\left(2\H\left(\frac{1}{\sqrt{1+y^2}}\right)\sqrt{1+y^2}-\frac{y}{x}\right)},\hspace{.5cm} x>r_0,\ y\in [0,1].
$$
Notice that Equation \eqref{hfsegunda} can be expressed as
$$
\begin{array}{c}
\vspace{.2cm}h''(x)=\phi_\H(x,h'(x)),\\
f''(x)=\phi_\F(x,f'(x)).
\end{array}
$$
As $\H(y)>\F(y)$, then is straightforward that
\begin{equation}\label{phihmayorf}
\phi_\H(x,y)>\phi_\F(x,y),\hspace{.5cm}\forall x>r_0,\ y\in [0,1].
\end{equation}
In this situation, the following inequality holds
$$
h''(x)=\phi_\H(x,h'(x))>\phi_\F(x,h'(x)).
$$
As $h'(x_0)>f'(x_0)$ and $f''(x)$ is a solution for the ODE $\phi_\F(x,f'(x))$, a classical comparison theorem for ODE's applied to Equation \eqref{phihmayorf} ensures us that $h'(x)>f'(x)$ for every $x>x_0$. 

The same proof works for the case of lower ends. Notice that in this situation, because the change of the orientation explained above, the mean curvatures are positives and thus the comparison of the ODE's in Equation \eqref{phihmayorf} now reads as $\phi_\H(x,y)<\phi_\F(x,y)$. Now the hypothesis $h'(x_0)<f'(x_0)$ and the comparison theorem for ODE's concludes the proof for lower ends. This completes the proof of Proposition \ref{monotoniaderivadas}.
\end{proof}

\begin{remark}
Suppose that $\H,F\in\c1$ satisfy $\H(y)>\F(y)$ for all $y\in (-1,1)$, and let be $r_0>0$. Then, we know that $\sig_\H(r_0)$ lies locally above $\sig_\F(r_0)$ near $r_0$. In particular, the functions $h_+,f_+$ defining the upper ends $\sig_\H^+(r_0)$ and $\sig_\F^+(r_0)$, respectively, satisfy $h_+(x)>f_+(x)$ and $h'_+(x)>f'_+(x)$, for $x\in (r_0,r_0+\varepsilon)$ where $\varepsilon>0$ is small enough. In virtue of Propositions \ref{comparisoncats} and \ref{monotoniaderivadas}, this behavior is fulfilled for every $x>r_0$, and not only in a neighborhood of $r_0$.

The same holds for the lower ends, after a change of the signs of the inequalities.
\end{remark}

\section{\large Comparison of the behavior at infinity of prescribed mean curvature catenoids}\label{sec3}
\vspace{-.5cm}
Once we have formulated some comparison results for the height and the derivative of $\H$-catenoids, we take care of the behavior at infinity of the $\H$-catenoids. First, we need to settle on the concept of when an $\H$-catenoid \emph{goes to infinity}.

\begin{defi}\label{defiboundedheight}
Let be $\hi$ and $r_0>0$, and consider the $\H$-catenoid $\sig_\H(r_0)$. Denote by $f_+$ and $f_-$ to the upper and lower height of $\sig_\H(r_0)$, respectively.
\begin{itemize}
\item[1.] We say that $\sig_\H(r_0)$ has \emph{unbounded upper height} (resp. \emph{bounded upper height}) if the function $f_+(x)$ is unbounded (resp. bounded).
\item[2.] We say that $\sig_\H(r_0)$ has \emph{unbounded lower height} (resp. \emph{bounded lower height}) if the function $f_-(x)$ is unbounded (resp. bounded).
\item[3.] If both $f_+(x)$ and  $f_-(x)$ are unbounded (resp. bounded), we will simply say that $\sig_\H(r_0)$ has \emph{unbounded height} (resp. \emph{bounded height}).
\end{itemize}
\end{defi}

The \emph{behavior at infinity} of an $\H$-catenoid $\sig_\H(r_0)$ is just the boundedness or unboundedness of its height functions. The following proposition proves that for a fixed $\H$, all the $\H$-catenoids $\sig_\H(r)$ have the same behavior at infinity.

\begin{pro}\label{mismocomportamientohcats}
Let be $\H\in\c1$ and $r_0>0$. Suppose that the $\H$-catenoid $\sig_\H(r_0)$ has unbounded (resp. bounded) upper height. Then, all the $\H$-catenoids $\{\sig_\H(r)\}_r$ have unbounded (resp. bounded) upper height. The same holds for the lower height.
\end{pro}

\begin{proof}
As usual, we present the proof for the upper height, since the lower height case is proved in the same way.

We start our proof with the case that the $\H$-catenoid $\sig_\H(r_0)$ has unbounded upper height.

First, consider some $r_*<r_0$ and define $\lambda=r_*/r_0<1$. Consider the homothety $\Phi_\lambda:\r3\rightarrow\r3,\ \Phi(p)=\lambda p$ for all $p\in\r3$, and let us define $\lambda\sig_\H(r_0):=\Phi_\lambda(\sig_\H(r_0))$. The mean curvature $H_{\lambda\sig_\H(r_0)}$ of $\lambda\sig_\H(r_0)$ satisfies $H_{\lambda\sig_\H(r_0)}=1/\lambda H_{\sig_\H(r_0)}$, and thus the uniqueness of the Cauchy problem yields
$$
\lambda\sig_\H(r_0)=\sig_{\frac{1}{\lambda}\H}(r_*).
$$
Note that an homothety does not change the behavior at infinity of an $\H$-catenoid, since it only multiplies by $\lambda\neq 0$ its upper height.

Consider also the $\H$-catenoid $\sig_\H(r_*)$. Thus, both $\sig_\H(r_*)$ and $\sig_{1/\lambda\H}(r_*)$ are tangent along their waists, and because $\lambda<1$ Proposition \ref{comparisoncats} ensures us that $\sig_\H(r_*)$ lies above $\sig_{1/\lambda\H}(r_*)$. As $\sig_\H(r_0)$ was supposed to have unbounded upper height, then $\sig_\H(r_*)$ has to have also unbounded upper height. On the contrary, at a finite point the $\H$-catenoid $\sig_{1/\lambda\H}(r_*)$ would intersect the $\H$-catenoid $\sig_\H(r_*)$, which would yield to a contradiction with Proposition \ref{comparisoncats}, see Figure \ref{homotecia}.
\begin{figure}[H]
\centering
\includegraphics[width=.8\textwidth]{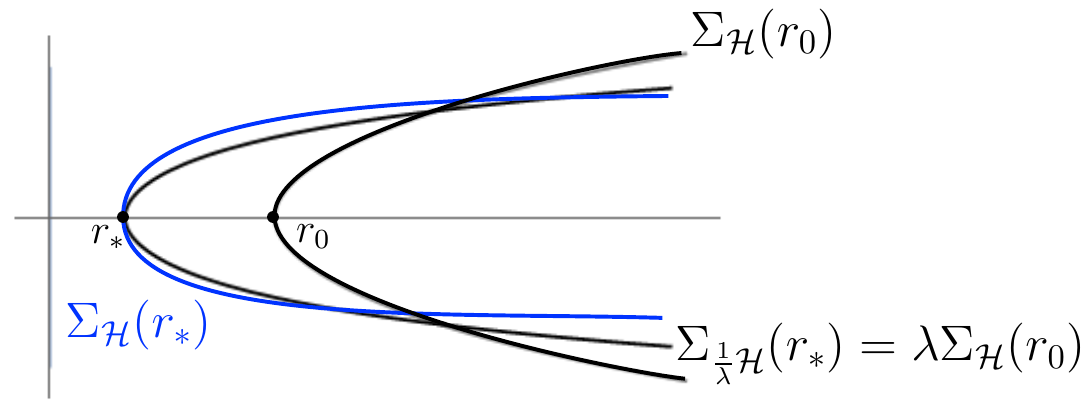}
\caption{The homothetical $\H$-catenoid $\lambda\sig_\H(r_0)$ has to stay below the $\H$-catenoid $\sig_\H(r_*)$.}
\label{homotecia}
\end{figure}

Suppose now that $r_*>r_0$. We will use the fact that all the behavior at infinity of the $\H$-catenoids $\{\sig_\H(r)\}_{r<r_0}$ agree with the behavior at infinity of $\sig_\H(r_0)$. When $r$ tends to zero, the sequence of $\H$-catenoids $\{\sig_\H(r)\}_{r}$ converges to a double covering of the plane $\{z=0\}$ minus the origin, see Proposition \ref{doblerecubri} in the Appendix. Thus, for $r_1$ close enough to zero, the $\H$-catenoid $\sig_\H(r_1)$ would intersect the $\H$-catenoid $\sig_\H(r_*)$ at the boundary of a slab $\mathcal{S}$, and such that $\sig_\H(r_1)$ lies below $\sig_\H(r_*)$ inside $\mathcal{S}$. Moving $\sig_\H(r_1)\cap\mathcal{S}$ upwards and downwards as in the proof of Proposition \ref{comparisoncats} we arrive to a contradiction with the maximum principle, see Figure \ref{homotecia1}. Thus, $\sig_\H(r_*)$ has to have unbounded upper height as well. 
\begin{figure}[H]
\centering
\includegraphics[width=.7\textwidth]{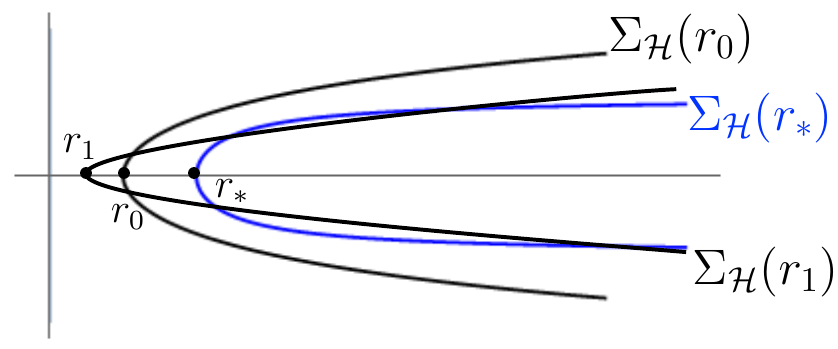}
\caption{The $\H$-catenoid $\sig_\H(r_*)$ would be intersected by an $\H$-catenoid $\sig_\H(r_1)$ for some $r_1$ small enough, contradicting the maximum principle.}
\label{homotecia1}
\end{figure}

From the discussions made above we may also conclude that if $\sig_\H(r_0)$ have bounded upper height, then all the $\H$-catenoids $\{\sig_\H(r)\}_{r>0}$ have bounded upper height. 

The same arguments work for the lower height, and thus Proposition \ref{mismocomportamientohcats} is proved.
\end{proof}

Let be $\hi$ and $r_0>0$. The phase plane study made in \cite{BGM} ensures us that the mean curvature of each $\H$-catenoid vanishes near infinity, and thus the angle function at the upper (resp. lower) end has to converge to $1$ (resp. to $-1$). This fact suggests us that the study of the $\H$-catenoids near infinity is closely related to the behavior of the prescribed function $\H$ at the points $y=\pm 1$.

First, we give a definition concerning two prescribed functions with the same limit behavior.

\begin{defi}\label{compasintontico}
Let be $\H,\F\in\c1$. 
\begin{itemize}
\item[1. ]We will say that $\H$ and $\F$ have the same \emph{behavior at $y=1$} if and only if
\begin{equation}\label{limy1}
\lim_{y\rightarrow 1}\frac{\H(y)}{\F(y)}=C_1,
\end{equation}
where $C_1$ is a nonzero constant. This condition will be denoted by $\H\sim_1\F$.

\item[2.] We will say that $\H$ and $\F$ have the same \emph{behavior at $y=-1$} if and only if
\begin{equation}\label{limym1}
\lim_{y\rightarrow -1}\frac{\H(y)}{\F(y)}=C_2,
\end{equation}
where $C_2$ is a nonzero constant. This condition will be denoted by $\H\sim_{-1}\F$.

\item[3.] If Equations \eqref{limy1} and \eqref{limym1} hold, we will say that $\H$ and $\F$ have the same \emph{behavior at $y=\pm 1$}. This condition will be denoted by $\H\sim\F$
\end{itemize} 
\end{defi}
A straightforward consequence from this definition is that each relation $\sim_1,\sim_{-1}$ and $\sim$ is an equivalence relation in the set of functions $\c1$. 

The fact that two functions $\H$ and $\F$ have the same behavior at either $\pm 1$ or both, is equivalent to the following: for each $\varepsilon\in (0,2)$ there exist nonzero constants $M,M'$ (depending on $\varepsilon$) such that
\begin{equation}\label{acotarconconstantes}
\begin{array}{llcl}
\mathrm{If}\ \vspace{.25cm}\H\sim_1\F, &\mathrm{then} & M\H(y)\leq\F(y)\leq M'\H(y),&\hspace{.5cm} \forall y\in [-1+\varepsilon,1],\\
\mathrm{If}\ \vspace{.25cm}\H\sim_{-1}\F, &\mathrm{then} & M\H(y)\leq\F(y)\leq M'\H(y),&\hspace{.5cm} \forall y\in [-1,1-\varepsilon],\\
\mathrm{If}\ \vspace{.25cm}\H\sim\F, &\mathrm{then}  & M\H(y)\leq\F(y)\leq M'\H(y),&\hspace{.5cm} \forall y\in [-1,1].
\end{array}
\end{equation}

The following proposition reveals that two functions with the same behavior at either $y=\pm 1$ generate classes of prescribed mean curvature catenoids with the same behavior at infinity.

\begin{pro}\label{hfmismocomportamiento}
Let be $\H,\F\in\c1$ and $r_0>0$, suppose that $\H\sim_1\F$ and consider the prescribed mean curvature catenoids $\sig_\H(r_0)$ and $\sig_\F(r_0)$. Then, the catenoid $\sig_\H(r_0)$ has bounded (resp. unbounded) upper height if and only if $\sig_\F(r_0)$ has bounded (resp. unbounded) upper height.

The same holds when $\H\sim_{-1}\F$ for their lower heights.
\end{pro}

\begin{proof}
We present the proof when $\H\sim_1\F$, since the case for the relation $\sim_{-1}$ is analogous.

Because $\H\sim_1\F$, Equation \eqref{acotarconconstantes} ensures us the existence of nonzero constants $M,M'$ such that
$$
M\F(y)\leq\H(y)\leq M'\F(y),\hspace{.5cm} \forall y\in [0,1].
$$

Consider the $\H$-catenoid $\sig_\H(r_0)$, and suppose first that $\sig_\H(r_0)$ has unbounded upper height. Now consider the catenoid $\sig_{M'\F}(r_0)$. Both $\sig_\H(r_0)$ and $\sig_{M'\F}(r_0)$ are tangent along their waists, and because $M'\F(y)>\H(y)$ for all $y\in [0,1)$, Proposition \ref{comparisoncats} ensures us that $\sig_{M'\F}(r_0)$ lies above $\sig_\H(r_0)$ always. As the upper height of $\sig_\H(r_0)$ is unbounded, the upper height of $\sig_{M'\F}(r_0)$ has to be also unbounded. Because $\sig_{M'\F}(r_0)$ and $\sig_\F(r_0)$ have the same behavior at infinity, we conclude that $\sig_\F(r_0)$ has also unbounded upper height. 

If the upper height of $\sig_\H(r_0)$ is bounded, we compare $\sig_\H(r_0)$ and the catenoid $\sig_{M\F}(r_0)$. 

The same comparison arguments hold for the case when $\H\sim_{-1}\F$. This proves Proposition \ref{hfmismocomportamiento}.
\end{proof}
Combining Propositions \ref{mismocomportamientohcats} and \ref{hfmismocomportamiento}, we state the following theorem 

\begin{teo}\label{claseequiv}
Let be $\H,\F\in\c1$, $r_0>0$ and consider the $\H$-catenoid $\sig_\H(r_0)$. Then,
\begin{itemize}
\item[1.] If $\H\sim_1\F$, then the upper end of each $\F$-catenoid has the same behavior at infinity as the upper end of $\sig_\H(r_0)$.
\item[2.] If $\H\sim_{-1}\F$, then the lower end of each $\F$-catenoid has the same behavior at infinity as the lower end of $\sig_\H(r_0)$ .
\item[3.] If $\H\sim\F$, then both ends of each $\F$-catenoid have the same behavior at infinity as the ends of $\sig_\H(r_0)$.
\end{itemize}
\end{teo}

\section{\large The behavior at infinity of some prescribed mean curvature catenoids}\label{sec4}
\vspace{-.5cm}
The arbitrariness of the prescribed function in Equation \eqref{ODEfsegunda} disables us to explicitly study the behavior at infinity for an arbitrary choice $\hi$. However, the study carried on in Section \ref{sec3} ensures us that the knowledge of the behavior at infinity for some family of $\H$-catenoids automatically reveals the behavior at infinity of the family of $\mathcal{F}$-catenoids, for each $\F\in\c1$ such that $\F\sim\H$.

Motivated by this fact, in this section we will study the family of $\H$-catenoids for some concrete choices of the prescribed function $\H$. Specifically, we will study the 1-parameter family of $\c1$ functions $\H_\alpha(y):=-(1-y^2)^\alpha,\ \alpha>1$. The theorem that we prove in this section is the following:

\begin{teo}\label{alfamayor1}
Let be $\alpha>1$ and consider the function $\H_\alpha(y)=-(1-y^2)^\alpha$. Then, for each $\F\in\c1$ we have:
\begin{itemize}
\item[1.] If $\F\sim_1\H_\alpha$, the $\F$-catenoids $\{\sig_\F(r)\}_{r>0}$ have unbounded upper height.
\item[2.] If $\F\sim_{-1}\H_\alpha$, the $\F$-catenoids $\{\sig_\F(r)\}_{r>0}$ have unbounded lower height.
\item[3.] If $\F\sim\H_\alpha$, the $\F$-catenoids $\{\sig_\F(r)\}_{r>0}$ have unbounded both upper and lower height.
\end{itemize}
\end{teo}

\begin{proof}
We will present the proof of Item 1, since Items 2 and 3 are proved similarly. 

Fix some $\alpha>1$. We will prove Item 1 of Theorem \ref{alfamayor1} by showing that for some $r_0$, the $\H_\alpha$-catenoid $\sig_{\H_\alpha}(r_0)$ has unbounded upper height. Then, in virtue of Theorem \ref{claseequiv} the upper ends of all the $\F$-catenoids $\{\sig_\F(r)\}_{r>0}$ for $\F\sim_1{\H_\alpha}$ will have the same behavior at infinity as $\sig_{\H_\alpha}(r_0)$. 

The study of the behavior at infinity of the $\H_\alpha$-catenoid $\sig_{\H_\alpha}(r_0)$ will be done by proving several claims.

\textbf{\emph{Claim 1.}} \emph{Consider the upper end of the $\H_\alpha$-catenoid $\sig_{\H_\alpha}(r_0)$ parametrized as the graph of a function $f(x)$. Then,}
$$
f'(x)<\frac{r_0}{\sqrt{x^2-r_0^2}},\hspace{.5cm} \forall x>r_0.
$$

\emph{Proof of Claim 1.} 
It is known that the upper end of the minimal catenoid $\mathcal{C}(r_0)$ is parametrized by the function
$$
g(x)=r_0\log\left(\frac{x+\sqrt{x^2-r_0^2}}{r_0}\right),\hspace{.5cm} x>r_0.
$$

At distance $x=r_0$, both $\mathcal{C}(r_0)$ and $\sig_{\H_\alpha}(r_0)$ are tangent along their waists, where their unit normals agree. The mean curvature comparison principle ensures us that $\mathcal{C}(r_0)$ lies above $\sig_{\H_\alpha}(r_0)$ around $x=r_0$. In particular, Proposition \ref{monotoniaderivadas} ensures us that $g'(x)>f'(x)$ for all $x>r_0$. Thus, the bound
$$
f'(x)<\frac{r_0}{\sqrt{x^2-r_0^2}},\hspace{.5cm} \forall x>r_0,
$$
holds by just substituting the value of $g'(x)$. Hence, Claim 1 is proved.
\nopagebreak\hfill $\Box$\endtrivlist

We now derive an inequality involving the derivative of the function $f(x)$. First, observe that for the prescribed choices $\H_\alpha$ , the value of $\H_\alpha(\nu_f(x))$ is given by
$$
\H_\alpha(\nu_f(x))=-\frac{f'(x)^{2\alpha}}{(1+f'(x)^2)^\alpha}.
$$
From Equation \eqref{ODEfsegunda}, we obtain the following
$$
-f'(x)^{2\alpha-1}-\frac{1}{x}<\frac{f''(x)}{f'(x)(1+f'(x)^2)}<-\frac{1}{2} f'(x)^{2\alpha-1}-\frac{1}{x}.
$$

Integrating from $r_0$ to $x>r_0$, we obtain
$$
-\int_{r_0}^xf'(t)^{2\alpha-1}dt+\log\frac{r_0}{x}<\log\left(\frac{f'(x)}{\sqrt{1+f'(x)^2}}\frac{\sqrt{1+f'(r_0)^2}}{f'(r_0)}\right)<-\frac{1}{2}\int_{r_0}^xf'(t)^{2\alpha-1}dt+\log\frac{r_0}{x}.
$$
After taking exponentials and some operations yields
\begin{equation}\label{cotafprima}
e^{\displaystyle{-\int_{r_0}^xf'(t)^{2\alpha-1}dt}}<\frac{xf'(x)}{\sqrt{1+f'(x)^2}}\frac{\sqrt{1+f'(r_0)^2}}{r_0f'(r_0)}< e^{\displaystyle{-\frac{1}{2}\int_{r_0}^xf'(t)^{2\alpha-1}dt}}.
\end{equation}

\textbf{\emph{Claim 2.}} \emph{Let be $\alpha>1$ and $f(x)$ the function that defines the upper end of $\sig_{\H_\alpha}$. Then,}
$$
\lim_{x\rightarrow\infty}xf'(x)=c_0,
$$
where $c_0$ is a positive constant.

\emph{Proof of Claim 2.}
Lets analyze the bound on $f'(x)$ in Equation \eqref{cotafprima}. Suppose that the integral $\int_{r_0}^xf'(t)^{2\alpha-1}dt$ is finite when $x$ goes to infinity and that the limit $\lim_{x\rightarrow\infty}xf(x)$ exists. Then, $\lim_{x\rightarrow\infty}xf(x)$ is necessarily a nonzero constant $c_0>0$.

Thus, our goal in this claim is twofold: firstly, we have to prove that the integral $\int_{r_0}^xf'(t)^{2\alpha-1}dt$ is finite when $x$ tends to infinity; secondly, we have to prove that the limit $\lim_{x\rightarrow\infty}xf(x)$ exists.

In order to prove that $\int_{r_0}^\infty f'(t)^{2\alpha-1}dt$ is finite, we consider the minimal catenoid $\C(r_0)$. In Claim 1 we proved that
$$
f'(x)<\frac{r_0}{\sqrt{x^2-r_0^2}},\hspace{.5cm} \forall x>r_0,
$$
and by powering to the $2\alpha-1$ we arrive to
$$
\lim_{x\rightarrow\infty} f'(x)^{2\alpha-1}<\left(\frac{r_0}{\sqrt{x^2-r_0^2}}\right)^{2\alpha-1},\hspace{.5cm} \forall x>r_0.
$$
Integrating from $r_0$ to $x$ yields
\begin{equation}\label{eqineq}
\int_{r_0}^x f'(t)^{2\alpha-1}dt<\int_{r_0}^x\left(\frac{r_0}{\sqrt{t^2-r_0^2}}\right)^{2\alpha-1}dt.
\end{equation}
Because $\alpha>1$ we have $2\alpha-1>1$, and thus the right hand side of Equation \eqref{eqineq} is a finite integral when $x$ diverges to $\infty$. This implies
$$
\int_{r_0}^x f'(t)^{2\alpha-1}dt<\infty.
$$
By making $x$ tend to infinity in Equation \eqref{cotafprima} we conclude that $xf'(x)$ is bounded between two positive constants. 

Now we will prove that $\lim_{x\rightarrow\infty}xf'(x)$ exists as a straightforward consequence from the fact that $xf'(x)$ is a monotonous function. Indeed, the derivative $(xf'(x))'$ is
$$
\begin{array}{c}
\vspace{.5cm}(xf'(x))'=f'(x)+xf''(x)=f'(x)+x\left(1+f'(x)^2\right)\left(2\H(\nu_f)\sqrt{1+f'(x)^2}-\displaystyle{\frac{f'(x)}{x}}\right)=\\
=x\left(1+f'(x)^2\right)2\H_\alpha(\nu_f)\sqrt{1+f'(x)^2}-f'(x)^3<0,
\end{array}
$$
where we have used that $f''(x)$ satisfies the ODE that appears in Equation \eqref{ODEfsegunda} and that $\H$ is negative and $f'(x)$ is positive.

As $xf'(x)$ is a monotonous function which is bounded between two positive constants, its limit must be a positive number, say $c_0$. This concludes the proof of Claim 2.
\nopagebreak\hfill $\Box$\endtrivlist

\textbf{\emph{Claim 3.}} \emph{The upper height of the $\H$-catenoid $\sig_\H(r_0)$ is unbounded.}

\emph{Proof of Claim 3.}
From Claims 1 and 2 we know that there exists some positive constant such that
$$
xf'(x)=c_0+h(x),
$$
where $h(x)$ is a positive function satisfying $\lim_{x\rightarrow\infty}h(x)=0$. Integrating $f'(x)$ yields
$$
f(x)=f(r_0)+c_0\log\frac{x}{r_0}+\int_{r_0}^x\frac{h(t)}{t}dt.
$$
Thus, the function $f(x)$ is unbounded and so the upper height of $\sig_\H(r_0)$ is unbounded, concluding the proof of Claim 3.
\nopagebreak\hfill $\Box$\endtrivlist

Because $\sig_{\H_\alpha}(r_0)$ has unbounded upper height, this property also holds in the equivalence class of $\H_\alpha$ defined by the relation $\sim_1$, in virtue of Theorem \ref{claseequiv}. This concludes the proof of Theorem \ref{alfamayor1}.
\end{proof}

\section{\large Half-space theorems for properly immersed $\H$-surfaces}\label{sec5}
\vspace{-.5cm}
In this last section we will obtain half-space theorems for properly immersed $\Hss$, by exploiting the study carried on in the previous sections concerning the behavior at infinity of the prescribed mean curvature catenoids.

First of all, we shall introduce some previous notation. We will denote by $\{z=c_0;\ c_0\in\R\}$ to the horizontal plane at height $c_0$. The \emph{upper} (resp. lower) horizontal half-space determined by $\{z=c_0;\ c_0\in\R\}$ is the open subset $\{z>c_0;\ c_0\in\R\}$ (resp. $\{z<c_0;\ c_0\in\R\}$). We define also the closed hemispheres by $\overline{\mathbb{S}^2_+}:=\mathbb{S}^2\cap\{z\geq 0\}$ and $\overline{\mathbb{S}^2_-}:=\mathbb{S}^2\cap\{z\leq 0\}$.

The main theorem of this paper is the following:

\begin{teo}\label{half-spaceth}
Let be $\H\in C^1(\mathbb{S}^2)$ and $M$ a connected, properly immersed, nonplanar $\H$-surface in $\r3$.
\begin{itemize}
\item[1.] If there exists $\F\in\c1$ such that $\H(x)\geq\F(\langle x,e_3\rangle)$ for every $x\in\overline{\mathbb{S}^2_+}$, and
\begin{equation}\label{igualnorte}
\lim_{y\rightarrow 1}\frac{\F(y)}{-(1-y^2)^\alpha}=C_1\neq 0
\end{equation}
holds for some $\alpha>1$, then $M$ cannot be contained in a lower half-space $\{z\leq c_0;\ c_0\in\R\}$.
\item[2.] If there exists $\F\in\c1$ such that $\H(x)\geq\F(\langle x,e_3\rangle)$ for every $x\in\overline{\mathbb{S}^2_-}$, and
\begin{equation}\label{igualsur}
\lim_{y\rightarrow -1}\frac{\F(y)}{-(1-y^2)^\beta}=C_2\neq 0
\end{equation}
holds for some $\beta>1$, then $M$ cannot be contained in an upper half-space $\{z\geq c_0;\ c_0\in\R\}$.
\item[3.] If Items 1 and 2 hold for some $\F\in\c1$ and some $\alpha,\beta>1$, then $M$ cannot be contained in any horizontal half-space.
\end{itemize}
\end{teo}

\begin{proof}
The proof of this theorem is based on the original ideas firstly introduced by Hoffman and Meeks \cite{HoMe}.

First, suppose that $\H\in C^1(\mathbb{S}^2)$ is a negative function only vanishing at the north pole, and suppose that Item 1 holds for some $\F\in\c1$ and $\alpha>1$. 

Arguing by contradiction, suppose that $M$ is a connected, properly immersed, nonplanar $\H$-surface contained in a half-space $\{z\leq c_0;\ c_0\in\R\}$. After a vertical translation we can suppose that $M$ is contained in the lower half-space determined by the plane $\Pi:=\{z=0\}$, but is not contained in any $\{z\leq -\varepsilon;\ \varepsilon>0\}$. First, notice that $M$ cannot intersect the plane $\Pi$. Indeed, if such intersection exists we would be able to find a point $p\in M\cap\Pi$, which is necessarily an interior tangent point. Denoting by $\eta_\Pi\equiv e_3$ and $\eta_M$ to the unit normals of $\Pi$ and $M$ respectively, then one of the following two cases must occur:
\begin{itemize}
\item[1.] At $p$ we have $(\eta_\Pi)_p=e_3=(\eta_M)_p$. Thus, both $M$ and $\Pi$ are $\Hss$ for the same prescribed function $\H$, and $\Pi$ lies locally above $M$ at $p$. The maximum principle \ref{ppiomax} ensures us that $M=\Pi$, contradicting the fact that $M$ is nonplanar.

\item[2.] At $p$ we have $(\eta_\Pi)_p=e_3=-(\eta_M)_p$, i.e. their orientations are opposite at $p$. Recall that $M$ is an $\Hs$ for the orientation induced by $\eta_M$, and that $\H$ is a negative function. In particular $H_M(p)\leq 0$, and by the previous item it has to be $H_M(p)<0$. Changing the orientation of $M$ we would obtain a $-\H$-surface with the opposite orientation $-\eta_M$. Thus, at $p$ we would have that $\Pi$ lies locally above $M$, and with this orientation $H_M(p)>0$. This is now a contradiction with the mean curvature comparison principle \ref{meancurvprinc}, since $\Pi$ is a planar surface that lies above $M$ around $p$ and $M$ has positive mean curvature at $p$ w.r.t. the orientation induced by $-\eta_M$.
\end{itemize}
In any case we have that $M\cap\Pi=\varnothing$, and thus $M$ is strictly contained in the half-space $\{z<0\}$.

Because $M$ is a proper surface which does not intersect the plane $\Pi$, the origin $\textbf{o}$ cannot be an accumulation point of $M$ and thus there exists an Euclidean ball $B(\textbf{o},R_0),\ R_0>0$ which is disjoint from $M$.

Now consider the family of $\F$-catenoids $\{\sig_\F(r)\}_{r>0}$. In virtue of Equation \eqref{igualnorte}, Theorem \ref{claseequiv} and Theorem \ref{alfamayor1}, we ensure that the family of $\F$-catenoids have unbounded upper height. In particular, the upper end of $\sig_\F(r)$, which we will denote by $\sig_\F^+(r)$, is a strictly concave graph with unbounded height w.r.t. the plane $\Pi$, and this holds for every $r>0$. Notice that for every $r<R_0$ we have $\left(\sig_\F^+(r)\cap\Pi\right)\subset B(\textbf{o},R_0)$, and in particular $\partial\sig_\F^+(r)\subset B(\textbf{o},R_0)$, for every $r<R_0$.

Again, as $M$ is proper there exists $\varepsilon>0$ small enough such that the vertical translation $M^+:=M+\varepsilon(0,0,1)$ does not intersect the ball $B(\textbf{o},R_0)$. From its definition it is clear that $M^+$ lies below the half-space $\{z<\varepsilon\}$.

Now consider the family of upper ends $\{\sig_\F^+(r)\}_{R_0>r>0}$. The boundary of each $\sig_\F^+(r_0)$ is contained in $B(\textbf{o},R_0)$, and thus they do not intersect $M^+$. Moreover, if the parameter $r$ converges to zero, the upper ends $\sig_\F^+(r)$ converge on compact sets to a covering of the plane $\{z=0\}$ minus the origin $\textbf{o}$, see Proposition \ref{doblerecubri}. By continuity and because $\sig_\F^+(r_0)$ has unbounded height w.r.t. the plane $\Pi$, there has to exist a finite, first interior contact point $p_0$ between some $\sig_\F^+(r_0),\ r_0>0$ and $M^+$.
\begin{figure}[H]
\centering
\includegraphics[width=.8\textwidth]{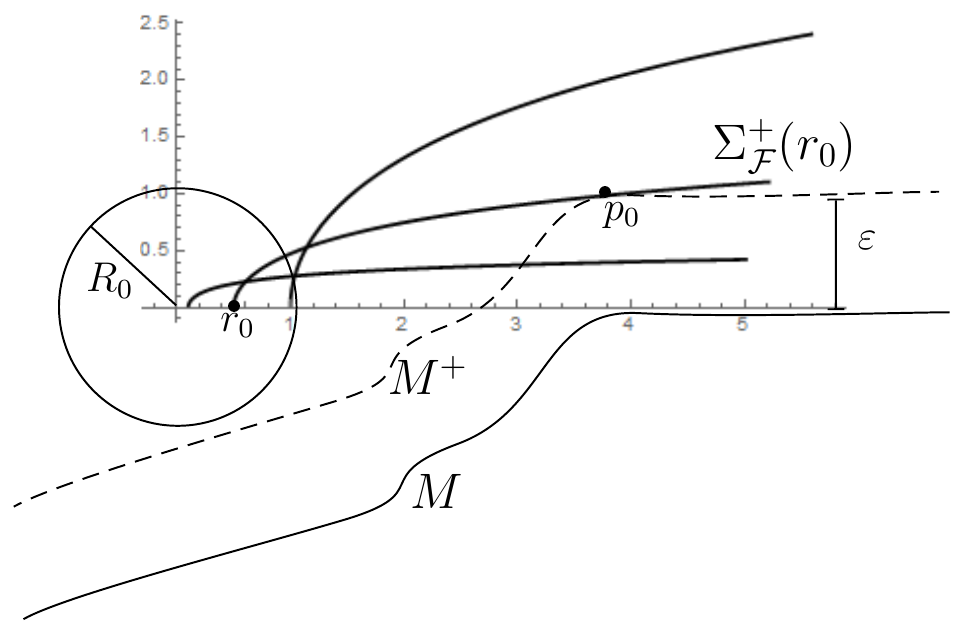}
\caption{By continuity there has to exist a first interior contact point $p_0$ between some $\sig_\F^+(r_0)$ and $M^+$.}
\label{halfspace}
\end{figure}
Now we argue as follows:

\begin{itemize}
\item[1.] If $(\eta_{\sig_\F^+(r_0)})_{p_0}=(\eta_{M^+})_{p_0}$, then $\sig_\F^+(r_0)$ lies locally above $M^+$ around $p_0$. But at $p_0$ we have
$$
H_{\sig_\F^+(r_0)}(p_0)=\F((\eta_{\sig_\F^+(r_0)})_{p_0})\leq\H((\eta_{\sig_\F^+(r_0)})_{p_0})=\H((\eta_{M^+})_{p_0})=H_{M^+}(p_0).
$$
contradicting the mean curvature comparison principle \ref{meancurvprinc}, since $\sig_\F^+(r_0)$ lies locally above $M^+$ around $p_0$.

\item[2.] If $(\eta_{\sig_\F^+(r_0)})_{p_0}=-(\eta_M)_{p_0}$, then we change the orientation of $M$, obtaining a $-\H$-surface with the opposite orientation $-\eta_M$. Thus, $\sig_\F^+(r_0)$ would lie locally above $M^+$ around $p_0$. But $\sig_\F^+(r_0)$ has negative mean curvature at $p_0$, and $M^+$ has positive mean curvature at $p_0$ after this change of orientation. This is again a contradiction with the mean curvature comparison principle \ref{meancurvprinc}.
\end{itemize}
In any case, we arrive to a contradiction and thus $M$ cannot be contained in a lower half-space.

The same arguments apply for the case that Item 2 holds, by just comparing $M$ with the lower ends of $\sig_\F(r)$.

The latter case, i.e. when Items 1 and 2 hold, is a straightforward consequence from the two previous cases.

This proves Theorem \ref{half-spaceth}.
\end{proof}

\begin{obs}
Note that we are only interested about the behavior of the function $\H$ near the points $N=(0,0,1)$ and $S=(0,0,-1)$. Thus, Theorem \ref{half-spaceth} is still valid if we are able to find a negative, 1-dimensional function $\F$ satisfying Equations \eqref{igualnorte} and \eqref{igualsur}, and such that $\H(x)\geq\F(\langle x,e_3\rangle)$ for every $x$ in a small neighborhood of the points $N,S$ (where $\F$ is defined). Then, we would be able to extend the function $\F$ to the interval $[-1,1]$ in such a way that its extension is smaller than $\H$ in the whole sphere.
\end{obs}

\section{Appendix}

This appendix is devoted to prove that the $\H$-catenoids converge to a double covering of the plane minus the origin. The proof will be done by showing that the $\H$-catenoids, outside a compact set that converges to the origin, have uniformly bounded second fundamental form, and then we will be able to take limits as a straightforward consequence of a compactness argument.
 
\begin{pro}\label{formulassf}
Let be $\hi$ and $r_0$, and consider $\sig_\H(r_0)$ an $\H$-catenoid. Then, the squared norm of the second fundamental form of the upper end $\sig_\H^+(r_0)$ is given by
\begin{equation}\label{normassf}
|\sigma_{\sig_\H^+(r_0)}(x)|^2=4\H(\nu_{\sig_\H(r_0)}(x))^2+\frac{2\sqrt{1-\nu_{\sig_\H(r_0)}(x)^2}}{x}\left(\frac{2\sqrt{1-\nu_{\sig_\H(r_0)}(x)^2}}{x}-2\H(\nu_{\sig_\H(r_0)}(x))\right).
\end{equation}
\end{pro}

\begin{proof}
We will derive this identity by computing the principal curvatures of the upper end $\sig_\H^+(r_0)$ parametrized as in Equation \eqref{pararot}. 

Indeed, this parametrization is doubly orthogonal and the principal curvatures already appeared in Equation \eqref{condicionHsup}; they are just the two terms on the r.h.s. of the equality. Thus,
$$
\begin{array}{c}
\vspace{.5cm}\kappa_1(x)=\displaystyle{\frac{{f_+}''(x)}{(1+{f_+}'(x)^2)^{3/2}}},\\
\kappa_2(x)=\displaystyle{\frac{{f_+}'(x)}{x\sqrt{1+{f_+}'(x)^2}}},
\end{array}
$$
where $f_+:(r_0,\infty)\rightarrow\R$ is the function that defines the upper end $\sig_\H^+(r_0)$ in Equation \eqref{pararot}.

Because $|\sigma_{\sig_\H(r_0)}(x)|^2=\kappa_1(x)^2+\kappa_2(x)^2$, Equation \eqref{normassf} yields by substituting the values of $\kappa_1(x)$ and $\kappa_2(x)$, and using that $f_+''(x)$ is a solution of the ODE \eqref{ODEfsegunda}.

This concludes the proof of Proposition \ref{formulassf}.
\end{proof}

Notice that for computing the quantity $|\sigma_{\sig_\H^-(r_0)}(x)|^2$ at the lower end $\sig_\H^-(r_0)$ with the parametrization given in Equation \eqref{pararot}, we have to change its orientation that was given by the unit normal defined in Equation \eqref{normalcontrario}, and the sign of the mean curvature. We omit the details.
 
\begin{pro}\label{doblerecubri}
Let be $\hi$, and consider the 1-parameter family of $\H$-catenoids $\{\sig_\H(r)\}_{r>0}$. Then, $\{\sig_\H(r)\}_{r\rightarrow 0}$ converges in the $C^3$ topology to a double covering of the plane $\{z=0\}$ minus the origin.
\end{pro}

\begin{proof}
The proof will be done as follows: for each $r>0$, we will find an open subset in $\sig_\H(r)$ with uniformly bounded second fundamental form, and such that this open subset does not contain the waist of $\sig_\H(r)$. At this point, a standard compactness argument for $\Hss$ obtained in \cite{BGM} will eventually conclude the result.

Thus, our main objective is to properly define an open subset in each $\sig_\H(r)$ having uniformly bounded second fundamental form. We will focus as usual in the upper end of $\sig_\H(r)$, since the computations for the lower end are similar. In order to save notation, we will omite the super index $(\cdot)^+$ referring to the upper end.

Let be $x_n=1/n,\ \nu_n=1-1/n^2,\ n\in\N$. The phase plane analysis done in \cite{BGM} ensures us that for each $(x_n,\nu_n)$ there exists an $\H$-catenoid $\sig_\H(r_n)$, with $r_n<x_n$ uniquely determined by $(x_n,\nu_n)$, and such that the angle function $\nu_{\sig_\H(r_n)}$ of $\sig_\H(r_n)$ satisfies $\nu_{\sig_\H(r_n)}(x_n)=\nu_n$.

Now, Equation \ref{normassf} ensures us that the squared norm of the second fundamental form of each $\sig_\H(r_n)$ is given by
$$
|\sigma_{\sig_\H(r_n)}(x)|^2=4\H(\nu_{\sig_\H(r_n)}(x))^2+2\frac{\sqrt{1-\nu_{\sig_\H(r_n)}(x)^2}}{x}\left(\frac{\sqrt{1-\nu_{\sig_\H(r_n)}(x)^2}}{x}-2\H(\nu_{\sig_\H(r_n)}(x))\right).
$$

Fix some $n\in\N$. If $x>y>r_n$, then $\nu_{\sig_\H(r_n)}(x)>\nu_{\sig_\H(r_n)}(y)$ and thus 
$$
\H(\nu_{\sig_\H(r_n)}(x))<\H(\nu_{\sig_\H(r_n)}(y)).
$$
In particular, we have
\begin{equation}\label{cota2ff1}
|\sigma_{\sig_\H(r_n)}(x)|^2<|\sigma_{\sig_\H(r_n)}(x_n)|^2,\hspace{.5cm} \forall x>x_n.
\end{equation}

If we compute $|\sigma_{\sig_\H(r_n)}(x_n)|^2$, we get
$$
|\sigma_{\sig_\H(r_n)}(x_n)|^2= 4\H(1-1/n)^2+2\sqrt{2-1/n^2}\left(\sqrt{2-1/n^2}-2\H(1-1/n)\right),\hspace{.5cm} \forall n\in\N.
$$
Because $\H$ is a negative function vanishing at $y=1$, we can consider the bound $-\H(1-1/n)<1$ for $n>n_0$, where $n_0\in\N$ is big enough. Bearing this in mind, the following estimate holds
\begin{equation}\label{cota2ff2}
|\sigma_{\sig_\H(r_n)}(x_n)|^2<4(2+\sqrt{2}),\hspace{.5cm} \forall n>n_0\in\N.
\end{equation}
Now, plugging together Equations \eqref{cota2ff1} and \eqref{cota2ff2} yields
\begin{equation}\label{cota2ffunif}
|\sigma_{\sig_\H(r_n)}(x)|^2<4(2+\sqrt{2}),\hspace{.5cm} \forall n>n_0\in\N,\ \forall x>x_n.
\end{equation}

Now we argue similarly in the lower ends $\sig_\H^-(r)$ to also obtain a uniformly bound of their second fundamental forms.

Thus, for every $n>n_0$ the squared norm of the second fundamental form of the $\H$-catenoid $\widetilde{\sig_\H(r_n)}=\sig_\H(r_n)\cap\{x\geq x_n\}$ is uniformly bounded. At this point, a standard compactness argument for $\Hss$, see e.g. Theorem in \cite{BGM} ensures us that the sequence $\widetilde{\sig_\H(r_n)}$ smoothly converges in the $C^3$ topology to a double covering of the the plane $\{z=0\}$ minus the origin.

This proves Proposition \ref{doblerecubri}.
\end{proof}

 \def\refname{References}

\end{document}